\documentclass[EJP]{ejpecp} 

\SHORTTITLE{The Symbiotic Contact Process} 

\TITLE{The Symbiotic Contact Process\thanks{Partially supported by DMS grants DMS 1505215 and 1809967 from the Probability Program. }} 

\AUTHORS{%
  Rick Durrett\footnote{Dept. of Math,  Duke University, Box 90320, Durham NC 27708-0320
    \EMAIL{rtd@math.duke.edu}}
  \and 
  Dong Yao\footnote{Dept. of Math,  Duke University, Box 90320, Durham NC 27708-0320 
  	\BEMAIL{dong.yao@duke.edu} }}


\AMSSUBJ{60K35} 

\SUBMITTED{April 4, 2019} 
\ACCEPTED{December 13, 2020} 

\VOLUME{0}
\YEAR{2012}
\PAPERNUM{0}
\DOI{vVOL-PID}

\ABSTRACT{    We consider a contact process on $\ZZ^d$ with two species that interact in a symbiotic manner. Each site can either be vacant or occupied by individuals of species $A$ and/or $B$. Multiple occupancy by the same species at a single site is prohibited. The name symbiotic comes from the fact that if only one species is present at a site then that particle dies with rate 1 but if both species are present then the death rate is reduced to $\mu \le 1$ for each particle at that site. We show the critical birth rate $\lambda_c(\mu)$ for weak survival is of order $\sqrt{\mu}$ as $\mu \to 0$. Mean-field calculations predict that when $\mu < 1/2$ there is a discontinuous transition as $\lambda$ is varied. In contrast, we show that, in any dimension, the phase transition is continuous. To be fair to the physicists that introduced the model, \cite{OSD}, the authors  say that the symbiotic contact process is in the directed percolation universality class and hence has a continuous transition. However, a 2018 paper, \cite{Fetal}, asserts that the transition is discontinuous above the upper critical dimension, which is 4 for oriented percolation.}

\begin{document}


 
\section{Introduction}

In the ordinary contact process the state at time  $t$ is a function $\xi_t : \ZZ^d \to \{0,1\}$. 1's are particles and 0's are empty sites. Particles die at rate 1, and are born at vacant sites at rate $\lambda f_1$ where $f_1$ is the number of nearest neighbors in state 1. A number of contact processes with two types of particles have been investigated. Neuhauser \cite{CN} considered the competing contact process $\xi_t : \ZZ^d \to \{0,1,2\}$. 0's again are vacant sites but now 1's and 2's are two types of particles. Each type of particle dies at rate 1, while particles of type $i$ are born at vacant sites at rate $\lambda_i f_i$ where $f_i$ is the \xyz{fraction} of nearest neighbors in state $i$,  $i=1,2$ . She showed that there was no coexistence if $\lambda_1 \neq \lambda_2$. When $\lambda_1=\lambda_2$ there is no coexistence in $d \le 2$ but there is when $d \ge 3$, behavior reminiscent \xyz{of} the voter model.

Durrett and Swindle \cite{DS} studied a contact process in which 0's are vacant sites, 1's are bushes and 2's are trees. 1's and 2's die at rate 1. Particles of type $i$ give birth at rate $\lambda_i$, and send their offspring to a randomly chosen nearest neighbor. If the site is vacant then it becomes type $i$. A 2 landing on a 1 changes the site to state 2, but a 1 landing on a 2 does nothing. In contrast to Neuhauser's model, coexistence is possible. More work on this system can be found in \cite{DM}.  

Krone \cite{K2s} studied a contact process with states 0, 1, and 2. Again 0's are vacant sites, but now 1's are young particles that cannot reproduce, while 2's are mature particles that can. Particles of type $i$ die at rate $\delta_i$. Transitions from 1 to 2 occur at a constant rate $\beta$. Vacant sites change to state 1 at rate $\lambda f_2$. Krone proved the existence of a phase transition and established some qualitative results about the phase diagram. He left a number of open problems, most of which were solved by Foxall \cite{F2s}.

\xyz{
	Lanchier and Zhang \cite{LY} studied the stacked contact process, which was then generalized by Foxall and Lanchier \cite{FL}.
	In the (generalized) stacked contact process the state space at each site is $ \{0,1,2\}$. 0 stands for empty site. 1 means there is a host but there is no symbiont  associated to the host.
	2 means there are both a host and an associated symbiont.
	0 becomes 1 at rate $\lambda_{10}f_1$ and  becomes 2 at rate 
	$\lambda_{20}f_2$. 2 and 1 becomes 0 at rate 1. 
	The symbiont is called a pathogen if $\lambda_{20}<\lambda_{10}$
	and a mutualist if $\lambda_{20}>\lambda_{10}$.
	1 becomes 2 at rate $\lambda_{21}f_2$ and 2 becomes 1 at rate $\delta$. The authors showed in \cite{FL} that in the case where $d=1$, $\delta=0$ and $\lambda_{10}>\lambda_c(\ZZ^1)>\lambda_{20}$ 
	only the host can survive locally but not the symbiont, no matter how large $\lambda_{21}$ is. 
	Here $\lambda_c(\ZZ^1)$ is the critical value of contact process in $\ZZ^1$. 
	This is
	in contrast with the mean field prediction which says there is coexistence of host and symbiont if 
	$\lambda_{10}$ and 
	$\lambda_{21}$ is large enough.
}

de Olivera, dos Santos, and Dickman \cite{OSD} introduced the symbiotic contact process  $\xi_t : \ZZ^d \to \{0,A,B,AB\}$. Here $A$'s and $B$'s are different species of particles. As in the contact process there can be at most one individual of a given type at a site, but in this process there can be one $A$ and one \xyz{$B$ at} a site. If only one type is present then the system reduces to a contact process in which particles die at rate 1, and vacant sites become occupied at rate $\lambda f_k$ where $f_k$ is \xyz{the} fraction of neighbors in state $i$.  Presence of the other type does not affect the birth rates, but the death rates of particles at doubly occupied sites is reduced to $\mu \le 1$ due to the symbiotic interaction between the two species.   

To describe the system formally, \xyz{for any site $x\in \ZZ^d$} we write the state of $x$ as $(i,j) \in \{0,1\}^2$ where $i$ is the number of individuals of species $A$ at the site and $j$ is the number of individuals of species $B$. Letting $f_A$ (resp.~$f_B$) be the fraction of neighbors that have an $A$ particle (resp.~a $B$ particle), \xyz{and $i,j \in \{0,1\}$} the transition rates are as follows

\begin{center}
	\begin{tabular}{cl}
		$(0,j) \to (1,j)$ & at rate $\lambda f_A$ \\
		$(i,0) \to (i,1)$ & at rate $\lambda f_B$ \\
		$ (1,0) \to (0,0) \quad (0,1) \to (0,0)$ & at rate 1 \\
		$(1,1) \to (0,1) \quad (1,1) \to (1,0)$ &at rate $\mu$
	\end{tabular}
\end{center}

\subsection{Mean-field calculations}

Often the first step in understanding the behavior of an interacting particle is to look at the predictions of mean-field theory in which we pretend that adjacent sites are independent. Let $p_0$, $p_A$, $p_B$ and $p_{AB}$ be the probabilities a site is in state $(0,0)$, $(1,0)$, $(0,1)$, and $(1,1)$. By considering the possible transitions we see that
\begin{align}
\frac{dp_A}{dt} & = \lambda p_0 (p_A+p_{AB}) + \mu p_{AB} - p_A -\lambda p_A(p_B + p_{AB})
\label{MF1}\\
\frac{dp_{AB}}{dt} & = 2\lambda p_A p_B + \lambda (p_A+p_B) p_{AB} - 2\mu p_{AB} 
\label{MF2}
\end{align}
See (1) in \cite{OSD}.  
We are only interested in solutions with $p_A=p_B=p$. \eqref{MF2} implies that in equilibrium we must have
$$
2\lambda p^2 + 2\lambda p p_{AB} - 2\mu p_{AB} = 0.
$$
Solving gives
\beq
p_{AB} = \frac{\lambda p^2}{\mu - \lambda p}.
\label{MF3}
\eeq
Noting that $p_0=1-p_A-p_B-p_{AB}$ and rewriting \eqref{MF1}
\beq
0 = \lambda(1-3p-p_{AB})(p+p_{AB}) + \mu p_{AB} - p
\label{MF4}
\eeq
Combining equations \eqref{MF3} and \eqref{MF4} we arrive at a quadratic equation for $p$ with coefficients that are quadratic polynomials in $\mu$ and $\lambda$.
Solving, see Section \ref{sec:mfc}, leads to the conclusions 

\begin{itemize}
	\item  For $\mu \ge 1/2$, $p$ grows continuously from zero at $\lambda=1$.
	\item When $\mu < 1/2$ there is a discontinuity at $\lambda = \sqrt{4\mu(1-\mu)}$.
\end{itemize}

\begin{figure}[tbp] 
	\centering
	\includegraphics[bb=53 57 738 555,width=4in,height=2.91in,keepaspectratio]{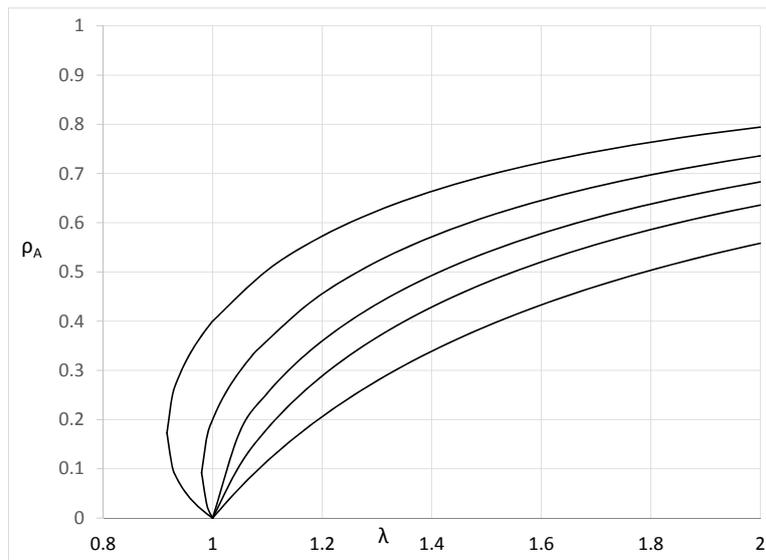}
	\caption{Plot of the equilibrium values of $\rho_A = p_A + p_{AB}$ versus $\lambda$ for $\mu = 0.3, 0.4, 0.5, 0.6, 0.8$. Curves increase as $\mu$ decreases. The curves for $\mu = 0.3$ and 0.4 show bistability when $\sqrt{4\mu(1-\mu)} < \lambda < 1$. \xyz{Here bistability refers to the fact that there are two attracting fixed point solutions to equations \eqref{MF1} and \eqref{MF2}}}. 
	\label{fig:MeanField}
\end{figure}

\subsection{Bounds on the critical value}\label{sec:bocv}

\xyz{Before defining critical value, we introduce a graphical representation for the symbiotic contact process (SCP). For each ordered pair of neighboring sites $x,y\in \ZZ^d$ we let $I^{x,y,A}$ be a Poisson process with rate $\lambda/2d$. When the Poisson clock associated to $(x,y)$ rings, we draw an arrow from $x$ to $y$ to indicate a possible infection events from $x$ to $y$, i.e., if there is an $A$ or $AB$ at $x$ at this time and there is no $A$ or $AB$ at $y$ then 
	the state at $y$ will change from 0 to $A$ or from $B$ to $AB$. Simirly we introduce the Poisson processes $I^{x,y,B}$ to indicate the possible infection events of species $B$. In what follows, we will frequently use the term  `arrow' to refer to possible infection events.}

\xyz{ Throughout the paper we will suppose $\mu \le1$. To represent the death events, we use Poisson processes
	$D^{x,A,1}$ with rate $\mu$ to denote death events that kills species $A$ at site $x$ no matter whether $B$ is present at $x$ at that time or not. 
	We also introduce Poisson processes $D^{x,A,2}$ with rate $1-\mu$ to denote the death events that kill species $A$ at site $x$ only if $B$ is not present at $x$. Similarly we also introduce two Poisson processes $D^{x,B,1}$ and $D^{x,B,2}$ to produce death events for species $B$.
	All the Poisson processes in the construction are independent.
}

Let $A_t$ be the number of sites that have an $A$, and let $B_t$ be the number of sites that have a $B$.
Let
\beq
\Omega_\infty = \{ A_t > 0 \hbox{ and } B_t > 0 \mbox{ for all } t\}
\label{omeginf}
\eeq
be the event that the process survives. \xyz{Using the graphical representation we can construct a coupling between SCP with $\mu_1=\mu_2$ and  $\lambda_1\leq \lambda_2$ by introducing new Poisson processes with rate $\lambda_2-\lambda_1$ that produce additional infections in the second process but not in the first. It is easy to see that if $\Omega_\infty$ happens in the first process then it also happens in the second. Let  $P_{AB,0}( \Omega_\infty )$ be the probability of survival when we start with one $AB$ at the origin and all other sites vacant. It follows from the coupling  that  $P_{AB,0}( \Omega_\infty )$
	is an nondecreasing function of $\lambda$ for fixed $\mu$. Hence we can define the critical value 
	\beq
	\lambda_c(\mu) = \inf\{ \lambda : P_{AB,0}( \Omega_\infty ) > 0 \}.
	\label{lamcr}
	\eeq
	If we set $\mu=1$ then the $A$'s and $B$'s are independent contact process. Using the graphical construction again we see that if $\mu \le 1$ then $\lambda_c(\mu) \le  \lambda_c(1)$ where $\lambda_c(1)$ is the critical value for survival of ordinary contact process on $\ZZ^d$.  To do this we build the process with $\mu$ using death Poisson processes $\hat D^{x,A,1}$ and  $\hat D^{x,B,1}$ that are a superposition of $D^{x,A,1}$ and $D^{x,A,2}$ and  $D^{x,B,1}$ and $D^{x,B,2}$ respectively.
}

The next result shows that symbiosis can have a great effect on the survival of the system. The upper bound uses a block construction which sacrifices accuracy to keep the renormalized sites independent, so it is very crude. As in the case of the ordinary contact process one can project $\ZZ^d$ into $\ZZ$ by mapping $x \to x_1 + \cdots + x_d$ in order to extend the result to $d>1$.

\begin{theorem}\label{bdscv}
	For any $d$, the critical value $\lambda_c(\mu)$ satisfies 
	$$
	C_1(\mu)\leq \lambda_c(\mu) \leq C_2(\mu).
	$$
\end{theorem}

\noindent
\xyz{The values of $C_1(\mu)$ and $C_2(\mu)$ are summarized in the following table.} 

\begin{center}
	\begin{tabular}{c|c|c|c}
		&  $C_1(\mu)$ 	& $C_2(\mu)$, $\mu < 1/1600$ & $C_2(\mu)$, $\mu \ge 1/1600$ \\
		\hline
		$d$=1 &   $\sqrt{8\mu-4\mu^2}$ & 40$\sqrt{\mu}$ & $\lambda_c(1)$ \\
		$d > 1$ & $\sqrt{\mu/2}-\mu/4$ & $\min\{40d\sqrt{\mu},\lambda_c(1)\}$ & $\lambda_c(1)$ 
	\end{tabular}
\end{center}

\mn
\xyz{To explain the division into cases, recall $\lambda_c(1)\geq 1$ in $d=1$ so  $\min\{40\sqrt{\mu}, \lambda_c(1)\}=40\sqrt{\mu}$ if $\mu<1/1600$.} 

\subsection{The phase transition is continuous}\label{sec:ptc}

We use a block construction for the basic contact process that is similar to the one originally developed by Bezuidenhout and Grimmett \cite{BezGr}. 
We follow the approach in Section \ref{sec:bocv} of Liggett \cite{TL}.  
For positive number $L$, let $I = [-(a+h)L, (a+h)L]$ and $W = I^d$. The constant $a$ is given in \eqref{acond}. 
We set $h = 1/2$ to make formulas easier to write. Let ${}_W\xi_t^{[-n,n]^d}$ denote the symbiotic contact process (SCP)  starting from all sites in $[-n,n]^d$ occupied by $AB$ and no births are allowed outside of $W$. Sometimes we omit the superscript $[-n,n]^d$. The space-time box is $B(L,T) = W \times [0,1.01T]$.
\xyz{$n$, $L$ and $T$ will be chosen in the proof of Theorem \ref{opstep1}. }
In the argument in Liggett one produces occupied translates of $[-n,n]^d$ on the top and sides of the space-time box. We will instead make copies near the top and near the sides in regions that we call slabs. The top slab is
$$
H_{L,T} = I^d \times [T,1.01T]
$$
We will have $2d$ side slabs. They are 
\begin{align*}
S^{i,+}_{L,T} & = I^{i-1} \times [hL, (a+h) L] \times I^{d-i} \times [0,1.1T] \\
S^{i,-}_{L,T} & = I^{i-1} \times [-(a+h)L, -hL] \times I^{d-i}  \times [0,1.1T]
\end{align*}
where $i$ ranges from 1 to $d$.

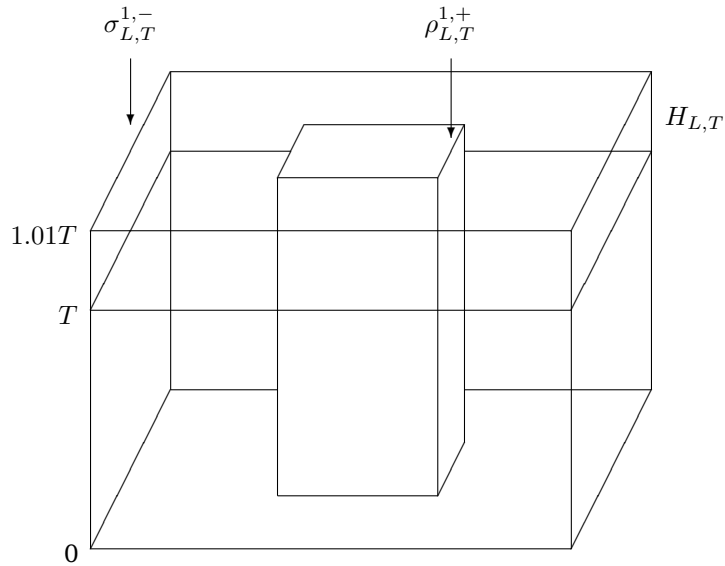
\begin{figure}[ht]
	\begin{center}
		\begin{picture}(270,240)
		\put(30,30){\line(1,0){180}}
		\put(30,30){\line(1,2){30}}
		\put(60,90){\line(1,0){40}}
		\put(240,90){\line(-1,0){70}}
		\put(210,30){\line(1,2){30}}
		\put(30,150){\line(1,0){180}}
		\put(30,150){\line(1,2){30}}
		\put(60,210){\line(1,0){180}}
		\put(210,150){\line(1,2){30}}
		\put(30,30){\line(0,1){120}}
		\put(20,25){0}
		\put(0,145){$1.01T$}
		\put(18,115){$T$}
		\put(60,90){\line(0,1){120}}
		\put(210,30){\line(0,1){120}}
		\put(240,90){\line(0,1){120}}
		\put(100,50){\line(1,0){60}}
		\put(160,50){\line(1,2){10}}
		\put(100,50){\line(0,1){120}}
		\put(170,70){\line(0,1){120}}
		\put(160,50){\line(0,1){120}}
		\put(100,170){\line(1,0){60}}
		\put(100,170){\line(1,2){10}}
		\put(110,190){\line(1,0){60}}
		\put(160,170){\line(1,2){10}}
		\put(155,225){$\rho^{1,+}_{L,T}$}
		\put(165,215){\vector(0,-1){30}}
		\put(35,225){$\sigma^{1,-}_{L,T}$}
		\put(45,215){\vector(0,-1){25}}
		\put(30,120){\line(1,0){180}}
		\put(30,120){\line(1,2){30}}
		\put(60,180){\line(1,0){45}}
		\put(240,180){\line(-1,0){70}}
		\put(210,120){\line(1,2){30}}
		\put(245,190){$H_{L,T}$}
		\end{picture}
		\caption{Picture of the block construction. To enhance visualization
			we have drawn the figure assuming that the solid with sides $\rho^{i,\pm}_{L,T}$ is opaque.}
	\end{center}
	\label{fig:block}
\end{figure}

The union of the side slabs $S_{L,T} = \cup_{i=1}^d (S^{i,+}_{L,T} \cup S^{i,-}_{L,T})$ is an ``annulus''
with (outer)  sides 
\begin{align*}
\sigma^{i,+}_{L,T} & = I^{i-1} \times \{(a+h) L\} \times I^{d-i} \times [0,1.1T] \\
\sigma^{i,-}_{L,T} & = I^{i-1} \times \{-(a+h)L\} \times I^{d-i}  \times [0,1.1T] \\
\sigma_{L,T} & = \cup_{i=1}^d (\sigma^{i,+}_{L,T} \cup \sigma^{i,-}_{L,T})
\end{align*}
and inner sides
\begin{align*}
\rho^{i,+}_{L,T} & = I^{i-1} \times \{h L\} \times I^{d-i} \times [0,1.1T] \\
\rho^{i,-}_{L,T} & = I^{i-1} \times \{-hL\} \times I^{d-i}  \times [0,1.1T] \\
\rho_{L,T} & = \cup_{i=1}^d (\rho^{i,+}_{L,T} \cup \rho^{i,+}_{L,T} )
\end{align*}

\begin{theorem}\label{opstep1}
	Suppose $\lambda_c(\mu) < \lambda_c(1)$ and that the SCP starting from a single $AB$ at the \xyz{origin} survives with positive probability and $\lambda < \lambda_c(1)$. For any $\ep>0$, there are choices of $a$ (only depending on $\lambda$), $ n, L, T$, s.t.
	\begin{equation}\label{topop}
	\P( {}_W\xi^{[-n,n]^d}_{t} = AB \hbox{ on } x+[-n,n]^d \hbox{ for some $(x,t) \in H_{L,T}$ }) \ge 1 - \ep 
	\end{equation}
	and for any $1\le i \le d$
	\begin{equation}\label{sideop}
	\P({}_W\xi^{[-n,n]^d}_{t}  = AB \hbox{ on } x+[-n,n]^d \hbox{ for some $(x,t)\in S^{i,+}_{L,T}$ })\geq 1-\ep.
	\end{equation}
	By symmetry the last result holds for $S^{i,-}_{L,T}$. 
\end{theorem}

\noindent
This is the analog of Theorem 2.12 in Liggett \cite{TL}. Once this is done one can repeat the comparison with oriented percolation described 
on pages 51--55 in \cite{TL} to show that the SCP dies out at the critical value. The restriction to $\lambda < \lambda_c(1)$ in Theorem \ref{opstep1} is needed because our proof uses the fact that the system with only one type of particle is subcritical.
It is natural to 

\begin{conjecture}\label{strineq}
	If $\mu<1$ then $\lambda_c(\mu) < \lambda_c(1)$
\end{conjecture} 

\noindent 
The bounds on the critical value given in Theorem \ref{bdscv} imply that $\lambda_c(\mu) < \lambda_c(1)$ for small $\mu$. Strict monotonicity results for critical values have been proved for percolation, the Ising model, and other related systems. For early results see Chapter 10 of Kesten's book on percolation \cite{Kesten}. Given a pair of lattices, ${\cal L}_1$ and ${\cal L}_2$, Menshikov \cite{MN} gave conditions that guaranteed that the site percolation critical values $p_c({\cal L}_1) > p_c({\cal L}_2)$. His results were later generalized in \cite{AG} by Aizenman and Grimmett, who showed that the critical value for an infinite entangled set of open bonds in $\ZZ^3$ is smaller than that the critical value for an infinite connected component of open bonds. They also showed for ferromagnetic spin systems that the critical temperature was a strictly increasing function of the interaction strengths. For more results in this direction see Bezuidenhout, Grimmett, and Kesten \cite{BGK}, or Balister, Bollob\'as and Riordan \cite{BBS}. Results for the Ising and Potts models are proved by reducing to dependent percolation using the Fortuin-Kasteleyn representation. In the analysis of percolation one uses that there is only one fluid moving through the graph, so we do not think these methods can be used to prove Conjecture \ref{strineq}, which involves two types of fluids spreading through a graphical representation.

In the case of the ordinary contact process, $\eta_t$, a second corollary of Theorem \ref{opstep1} is the complete convergence theorem. That is if $\tau = \inf\{ t: \eta_t = \emptyset \}$ then
$$
\eta_t \Rightarrow P( \tau < \infty ) \delta_0 + P(\tau = \infty) \eta^1_\infty
$$
Here $\Rightarrow$ is convergence in distribution, $\delta_0$ the point mass on the all 0's configuration and ${\eta^1_\infty}$ is the limit starting from all sites occupied. The SCP is attractive so the limit $\xi^{AB}_\infty$ starting from all $AB$ exists, but SCP process is not additive in the sense of Harris \cite{HA}, so there is no dual process, which is a key ingredient in the contact process proof. \xyz{Therefore one can not duduce complete convergence for SCP process from Theorem \ref{opstep1}. (It is not clear to the authors whether complete convergence holds or not.)}

The absence of an additive dual creates another open problem.
\xyz{We say $\xi^{AB}_{\infty}$ is nontrivial if it puts positive mass on configurations with infinitely many $A$'s and infinitely many $B$'s.
	Using the graphical representation for the SCP introduced in Section \ref{sec:bocv}
	one can show if $\xi^{AB}_{\infty}$  is nontrivial for some $\lambda_1$, then $\xi^{AB}_{\infty}$  must also be nontrivial for any $\lambda\geq \lambda_1$. Hence for fixed $\mu \le 1$ we can define another critical value.}
$$
\lambda_e(\mu) = \inf\{ \lambda :  \xi^{AB}_\infty \hbox{ is \xyz{nontrivial}} \}.
$$

Using the graphical representation one can show that
$$
\P(\xi^{AB}_{\infty}\neq \delta_0 )\geq \P(\Omega_{\infty}),
$$
where $\Omega_{\infty}$ has been defined in \eqref{omeginf}. It is the probability that SCP starting from a sgingle site survives. 
This implies that $\lambda_e \le \lambda_c$. Toom's model shows that these two critical values are not equal in general. In this model the state of the system is $\xi_t : \ZZ^2 \to \{0,1\}$. As in the contact process occupied sites (1's)  become vacant (0) at rate 1. However now a vacant site at $x$ becomes occupied at rate $\lambda$ \xyz{only if}  both $x+(1,0)$ and $x+(0,1)$ are occupied. For Toom's model finite configuration cannot escape from a box that contains it, \xyz{
	since if
	initially $\xi_0\in R$ for some rectangle $R$, then all $x$ with $x\notin R$ can't get infected since 
	either $x+(1,0)$ or $x+(0,1)$ is not in $R$. Since we have a death rate 1 for all occupied sites, there must be a finite time when all sites of $R$ becomes empty. The process then dies out.
} It follows that
$\lambda_c = \infty$. Toom \cite{T80} proved that $\lambda_e < \infty$. See Bramson and Gray \cite{BrGr} for another approach. For some rigorous results about this model see \cite{RDsex}.  

It would be interesting if the SCP was another example in which the two critical values are different, but we have no reason to believe they should be, so we

\begin{conjecture}
	$\lambda_e(\mu) = \lambda_c(\mu)$
\end{conjecture}

\subsection{Symbiotic contact process with diffusion}

The question we address here is ``What happens if, in additions to the birth and death events, we also let particles move according to the simple exclusion process?'' This question was considered by de Olivera and Dickman \cite{OD}. To be precise, we view the symbiotic contact process as taking place on $\ZZ^d\times \{0,1\}$ with $A$'s living on level 0 and $B$'s living on level 1. Particles jump to each neighbor on the same level at rate $\ep^{-2}$, subject to the exclusion rule: if the chosen neighbor is already occupied then nothing happens. We refer to this process as the symbiotic contact process with diffusion (SCPD) and denote it by $\xi^{\ep}_{t}$.

In \cite{OD} simulations showed that for moderate diffusion rates, the process exhibits discontinuous phase transition but the transition becomes continuous again once $\ep$ gets small enough. They conjecture the critical value for $\lambda$ for small $\ep$ is 1 regardless of the value of $\mu$. Theorem \ref{SCPD} supports their conjecture.

Consider the SCPD starting from all sites occupied by $AB$. By attractiveness $\xi^{\ep,AB}_{t}$ has a weak limit as $t\to\infty$ that we denote by  $\xi^{\ep,AB}_{\infty}$. Here, we are interested in the limit of $\xi^{\ep,AB}_{\infty}$ as $\ep\to 0$.  To do this using the methods of Durrett and Neuhauser \cite{DC}
it is convenient to implement the simple exclusion dynamics using the stirring process: for each pair of neighbors $x$ and $y$ on a given level we exchange the values at $x$ and $y$ on that level at rate $\ep^{-2}$. 
Following the approach in \cite{DC} the first step is to show convergence of a ``dual process'' to a branching Brownian motion and then derive a partial differential equation for the evolution of the local densities of $A$, $B$ and $AB$, which we denote by $p^{\ep}_A(t,x)$, $p^{\ep}_B(t,x)$, $p^{\ep}_{AB}(t,x)$. We also set $q^{\ep}_A(t,x)=p^{\ep}_{AB}(t,x)+p^{\ep}_A(t,x)$ and $q^{\ep}_B(t,x)=p^{\ep}_{AB}(t,x)+p^{\ep}_B(t,x)$. As $\ep\to 0$, $q^\ep_A \to q_A$ and $q^\ep_B \to q_B$ that satisfy:
\begin{align}
\frac{\partial q_A}{\partial t}=\frac{1}{2}\Delta q_A+q_A(\lambda(1-q_A) -1+(1-\mu)q_B) \label{evoqA}\\
\frac{\partial q_B}{\partial t}=\frac{1}{2}\Delta q_B+q_B(\lambda(1-q_B) -1 +(1-\mu)q_A) \label{evoqB}
\end{align}
If we use the initial condition $q^{\ep}_A(0,\cdot)=q^{\ep}_B(0,\cdot)=1$ then by symmetry we can replace $q_B$ by $q_A$ in equation \eqref{evoqA} to get
\begin{equation}\label{evoqA2}
\frac{\partial q_A}{\partial t}=\frac{1}{2}\Delta q_A+(\lambda-1)q_A -(\lambda-1+\mu)q_A^2.
\end{equation}
The reaction term on the right-hand side is $=0$ when $q_A=0$ or $q_A=(\lambda-1)/(\lambda-1+\mu)$. If $\lambda<1$ 
the reaction term is $<0$ on $(0,1]$ so the limit is $\equiv 0$.
When $\lambda>1$, we have $f'(0)=\lambda-1>0$, so the root at 0 is unstable.
Results of Aronson and Weinberger \cite{AW75,AW78} imply that \eqref{evoqA2} has a traveling wave solution and that starting from any
initial condition $q_A(0,x) \in [0,1]$ that is not identically 0, $q_A(t,x)$ is close to $(\lambda-1)/(\lambda+\mu-1)$ on a linearly growing set. 

The fast stirring makes the states on the $A$ and $B$ lattice independent so the equilibrium density of $AB$'s is
$$
p_{AB} =  \left( \frac{\lambda-1}{\lambda+\mu-1} \right)^2
$$ 
Subtracting this from the limit of the $q_A$ we find that the limiting density of $A$'s and $B$"s are given by
$$
p_A = p_B = \frac{(\lambda-1)\mu}{(\lambda+\mu-1)^2}
$$
Combining the PDE result with a block construction one can establish the existence of a nontrivial stationary distribution when $\lambda<1$ using the methods in \cite{DC}. However, if we use Theorem 1.4 in \cite{CDP} instead we get the stronger result:

\begin{theorem}\label{SCPD}
	Fix any $\mu$ $\in$ (0,1]. (i) If $\lambda > 1$, then for small $\ep>0$, $\xi^{\ep,AB}_{\infty}$ is nontrivial, and in any
	stationary distribution that assigns mass 1 to configurations with infinitely many $A$'s and $B$'s the densities of
	$A$'s, $B$'s, and $AB$'s are close to $p_A$, $p_B$ and $p_{AB}$. 
	
	\noindent   
	(ii) If $\lambda < 1$ then for small $\ep$ we have $\xi^{\ep,AB}_{\infty} \equiv 0$.
\end{theorem}
\xyz{
	The proof of Theorem \ref{SCPD} is described in Section \ref{sec:thm3}}. All of the steps are in \cite{DC} or \cite{CDP}, but for the convenience of the reader, we will give an outline of the argument and indicate where detailed proofs can be found. It may surprise the reader to hear that the proof of the second result is much more difficult than the first. In part (i) if we can prove that the densities are close to the proposed values then we can conclude there is a non-trivial stationary distribution. However, in part (ii) we have to show that the density is 0, not just close to 0.

\subsection{Outline of the paper}

The remainder of the paper is devoted to proofs. The mean-field calculations are carried out in Section \ref{sec:mfc}. The bounds on the critical values are proved in Section \ref{sec:bdscv}.  The proof of Theorem \ref{opstep1} fills Sections \ref{sec:pfL1} to \ref{sec:pfT2}. We will now state the two lemmas that are the key to its proof. However, first we need more notations.

\xyz{ Given a finite space-time rectangular solid ${\cal R}$, the number of $AB$'s in ${\cal R}$, $AB({\cal R})$, is the size of the largest set ${\cal C}$ of space time points $(x,t)$ in ${\cal R}$ so that (i) if $(x,t) \in {\cal C}$ then $\xi(x,t)=AB$ and (ii) if $(x_1,t_1), (x_2, t_2) \in {\cal C}$ then either $x_1\neq x_2$ or $x_1=x_2$ and $|t_1-t_2| \ge 1$. Since the size of such a collection is bounded above by the cardinality of a set that has a point at each $x$ every 1 unit of time, a maximal set exists.  Similarly we can define $A({\cal R}), B({\cal R}).$ }

For some $M, N>0$, whose specific values are to be determined, let 
$$
R_{L,T} = \{ AB(H_{L,T}) \le  M, AB(S_{T,L}) \le N \},
$$
\xyz{We omit the dependence of various events on parameters other than $L$ and $T$ to make notation simpler.
	Recall we start the SCP from $[-n,n]^d$  occupied by $AB$.  Also, there is a parameter $a$ involved in the definition of the block construction in Section \ref{sec:ptc}.} In the next two lemmas $\kappa(\lambda)$ is a constant introduced in equation \eqref{expdecay} and \eqref{expdecay2}.
Let $\Omega_0$ be the event that the SCP dies out.

\begin{lemma}\label{block_new}
	For any $T_j, L_j \to \infty$ with $$\lim_{j\to \infty} T_j(\kappa(\lambda))^{aL_j}L_j^{d-1}=0 \mbox{ and }
	\lim_{j\to\infty} T_j/(\log L_j) =\infty$$ we have   
	$$
	\limsup_{j\to \infty}\P(R_{L_j,T_j})\leq \P(\Omega_0)
	$$
\end{lemma}

\begin{lemma}\label{persis_new}
	For any sequence $L_j\to \infty$, one can choose $T_j=T(L_j)$ that satisfies 
	\begin{equation}\label{tj}
	\lim_{j\to \infty} T_j (\kappa(\lambda))^{aL_j}L_j^{d-1}=0 \mbox{ and } \liminf_{j\to\infty} \frac{T_j}{L_j}>0,
	\end{equation} 
\end{lemma}

These Lemmas are similar to steps in Liggett's proof. For example, Lemma \ref{block_new} is analogous to his Proposition 2.8. However, here we need a large number of sites occupied by $AB$'s in $S^{i,+}_{L_j,T_j}$, but the argument in \cite{TL} only gives us a large number of sites occupied by $A$'s or $B$'s. Intuitively since the SCP with only one type contact is subcritical an isolated $A$ will soon die. So if the SCP is to survive with high probability there must be many $AB$'s . To translate this idea into a proof, we show  \xyz{in Section \ref{sec:pfL2} that} an $A$ on the inner side $\rho^{i,+}_{L,T}$ is unlikely to have a descendant on the outer side $\sigma^{i,+}_{L,T}$ unless it encounters a $B$ along the way (and hence produces an $AB$).

Lemma \ref{block_new} is proved in Section \ref{sec:pfL1} and Lemma \ref{persis_new} in Section \ref{sec:pfL2}. These two results are combined in Section \ref{sec:pfT2} to prove Theorem \ref{opstep1}.  All of this material is independent of the calculations in Sections \ref{sec:mfc} and \ref{sec:bdscv}.

\clearp

\section{Mean-field calculations} \label{sec:mfc}

Using \eqref{MF3} we get
$$
0 = \lambda\left(1-3p-\frac{\lambda p^2}{\mu - \lambda p}\right)\left(p+\frac{\lambda p^2}{\mu - \lambda p}\right) 
+ \mu\frac{\lambda p^2}{\mu - \lambda p} - p.
$$ 
Multiplying by $(\mu-\lambda p)^2$ we have
\begin{align*}
0 & = \lambda((1-3p)(\mu-\lambda p) - \lambda p^2) \cdot (p(\mu-\lambda p)+\lambda p^2) 
+ \mu\lambda p^2(\mu - \lambda p) - p (\mu - \lambda p)^2 \\
& = \lambda(\mu -(\lambda+3\mu)p + 2\lambda p^2)(p\mu)
+ (\mu^2\lambda p^2 - \mu\lambda^2 p^3) - p (\mu^2 -2\mu\lambda p + \lambda^2 p^2) \\  
& = \lambda\mu^2 p - \lambda(\lambda+3\mu)\mu p^2 + 2 \lambda^2 \mu p^3 +\mu^2\lambda p^2 - \mu\lambda^2 p^3 - \mu^2 p +2\mu\lambda p^2 - \lambda^2 p^3.
\end{align*} 
Dividing by $p$ we arrive at the quadratic equation
$$
0 = (\lambda \mu^2 - \mu^2) + (-\lambda(\lambda+3\mu)\mu + \mu^2\lambda + 2\mu \lambda) p + (2\lambda^2 \mu -\mu \lambda^2 -\lambda^2) p^2,
$$
or after some algebra
\begin{align*}
0 & = \mu^2 (\lambda - 1) + (- \lambda^2\mu - 2\mu^2\lambda + 2\mu \lambda) p + \lambda^2 (\mu -1) p^2 \\
& = \mu^2 (\lambda  - 1) + \mu\lambda (2 - \lambda - 2\mu) p + \lambda^2 (\mu -1) p^2.
\end{align*}

Using the quadratic formula the solutions are
\begin{align*}
& \frac{-\mu\lambda[2(1-\mu)- \lambda] \pm \sqrt{ \mu^2\lambda^2[2(1-\mu)- \lambda]^2 - 4 \mu^2 \lambda^2 (\lambda-1)(\mu-1) }}
{2\lambda^2(\mu-1)} \\ 
& = \mu \frac{[2(1-\mu)- \lambda] \pm \sqrt{ [2(1-\mu)- \lambda]^2 - 4(\lambda-1)(\mu-1) }}
{2\lambda(1-\mu)} 
\end{align*}
Note that in the last step we removed the minus sign out front by changing $\mu-1$ to $1-\mu$ in the denominator. 
If $\lambda=1$ the second term under the square root vanishes
and the numerator is $2(1-\mu) - 1 \pm |2(1-\mu)-1|$. If $\mu \ge 1/2$ the plus root is 0 and the minus root is $<0$ so $\lambda_c=1$
and there will be one positive root. If $\mu< 1/2$ then the plus root is $>0$ and the minus root is 0. 

To further investigate the case $\mu < 1/2$ note that inside the square root
\begin{align*}
[2(1-\mu)- &\lambda]^2 - 4(\lambda-1)(\mu-1) \\
& = \lambda^2 - 4\lambda(1-\mu) + 4(1-\mu)^2 - 4 \lambda (\mu-1) + 4 (\mu-1)
\end{align*}
Cancelling the second and fourth terms gives
$$
= \lambda^2 + 4(1-2\mu+ \mu^2 + \mu - 1) = \lambda^2 - 4\mu(1-\mu)
$$ 
so the roots can be written as
$$
\frac{\mu}{2\lambda(1-\mu)} \left\{ 2(1-\mu) - \lambda \pm \sqrt{\lambda^2 - 4\mu(1-\mu)} \right\},
$$
which agrees with (5) in \cite{OD}.

If $\lambda < \sqrt{4\mu(1-\mu)}$ the roots are complex. 
When $\mu < 1/2$ and $\lambda = \sqrt{4\mu(1-\mu)}$, 
$$
2(1-\mu) - 2\sqrt{\mu(1-\mu)}>0,
$$ 
since $\mu < 1-\mu$. Combining this with the previous observation, we see that there are two roots when $4\mu(1-\mu) \le \lambda < 1$. 
The larger root and the root at 0 are stable so we have bistability for the ODEs
\eqref{MF1} and \eqref{MF2}
in this region. (\xyz{An ODE is bistable if it  has two attracting fixed points.})

\clearp  		

\section{Proof of critical vaule bounds} \label{sec:bdscv}

\subsection{\bf Upper bound on $\lambda_c$}

We first consider the case dimension $d=1$. Once we show $\lambda_c \le C \sqrt{\mu}$ in $d=1$ we can prove the result in $d>1$ by restricting the process to a line and it follows that $\lambda_c \le C d\sqrt{\mu}$ . We use a block construction. $(m,n) \in {\cal L}$ is said to be wet if one of the sites in $\{2m,2m+1\}$ is in state $AB$ at time $nT$, where $T$ is to be specified later. Let $B_{m,n} = [2m-1,2m+2] \times [nT,(n+1)T]$. \xyz{If} $n$ is even we only allow $m$ to be even integers and if $n$ is odd we only allow $m$ to be odd integers. Note that these blocks are disjoint. So if we only use arrows with both ends in the box to spread the occupancy, the events associated with these boxes are independent.

\begin{figure}[ht]
	\begin{center}
		\begin{picture}(300,120)
		\put(30,40){$\bullet$}
		\put(50,40){$\bullet$}
		\put(70,40){$\bullet$}
		\put(90,40){$\bullet$}
		\put(110,40){$\bullet$}
		\put(130,40){$\bullet$}
		\put(150,40){$\bullet$}
		\put(170,40){$\bullet$}
		\put(190,40){$\bullet$}
		\put(210,40){$\bullet$}
		\put(230,40){$\bullet$}
		\put(250,40){$\bullet$}
		\put(30,70){$\bullet$}
		\put(50,70){$\bullet$}
		\put(70,70){$\bullet$}
		\put(90,70){$\bullet$}
		\put(110,70){$\bullet$}
		\put(130,70){$\bullet$}
		\put(150,70){$\bullet$}
		\put(170,70){$\bullet$}
		\put(190,70){$\bullet$}
		\put(210,70){$\bullet$}
		\put(230,70){$\bullet$}
		\put(250,70){$\bullet$}
		\put(30,100){$\bullet$}
		\put(50,100){$\bullet$}
		\put(70,100){$\bullet$}
		\put(90,100){$\bullet$}
		\put(110,100){$\bullet$}
		\put(130,100){$\bullet$}
		\put(150,100){$\bullet$}
		\put(170,100){$\bullet$}
		\put(190,100){$\bullet$}
		\put(210,100){$\bullet$}
		\put(230,100){$\bullet$}
		\put(250,100){$\bullet$}
		\put(21,20){$-3$}
		\put(41,20){$-2$}
		\put(61,20){$-1$}
		\put(90,20){$0$}
		\put(110,20){$1$}
		\put(130,20){$2$}
		\put(150,20){$3$}
		\put(170,20){$4$}
		\put(190,20){$5$}
		\put(210,20){$6$}
		\put(230,20){$7$}
		\put(250,20){$8$}
		\put(60,43){\line(0,1){30}}
		\put(140,43){\line(0,1){30}}
		\put(220,43){\line(0,1){30}}
		\put(20,73){\line(0,1){30}}
		\put(100,73){\line(0,1){30}}
		\put(180,73){\line(0,1){30}}
		\put(260,73){\line(0,1){30}}
		\put(20,43){\line(1,0){240}}
		\put(20,73){\line(1,0){240}}
		\put(20,103){\line(1,0){240}}
		\put(270,40){$n=0$}
		\put(270,70){$n=1$}
		\put(270,100){$n=2$}
		\put(85,55){$m=0$}
		\put(165,55){$m=2$}
		\put(43,85){$m=-1$}
		\put(128,85){$m=1$}
		\put(205,85){$m=3$}
		\end{picture}
		\caption{Picture of the block construction.}
	\end{center}
\end{figure}
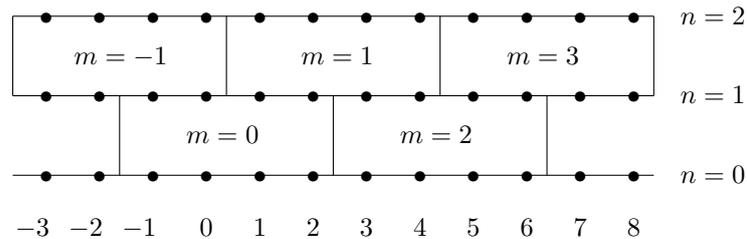

In order to compare with oriented site percolation on ${\cal L} = \{ (m,n) : m+n \hbox{ is even }\}$, we need an upper bound on the critical value. 
A simple contour argument, see e.g., Section 4 in \cite{StFl}, shows that when sites are open with probability $p > 80/81$ there is positive probability of percolation. Using a result of Balister, Bollob\'as, and Stacy \cite{BBS} gives an upper bound $p_c < 0.726$.
\xyz{Returning to the SCP, we can show}

\begin{lemma} \label{star}
	Suppose $\lambda>40\sqrt{\mu}$ and $\mu<1/1600$.
	If $(m,n)$ is wet then $(m-1,n+1)$ and $(m+1,n+1)$ will be wet with probability $> 0.726$.
\end{lemma}

We first compute for the process with $\mu=0$. 	Without loss of generality suppose $m=0$ and $n=0$ and that the $AB$ is at 0. The state of site 1 will go from 
\begin{align*}
0 \to A \hbox{ or } B &\quad\hbox{at rate } \lambda \\
A \hbox{ or } B \to 0 &\quad\hbox{at rate } 1 \\
A \hbox{ or } B \to AB &\quad\hbox{at rate } \lambda/2
\end{align*}
Let $N$ be the number of transitions  $0 \to A \hbox{ or } B$ before we arrive at $AB$. Since this is the \xyz{first} time the third transition occurs before the second one, it is clear that $N$ has a geometric distribution with success probability $(\lambda/2)/(1 + \lambda/2)$, that is, 
$$
\P(N=k) = \left(\frac{2}{\lambda +2} \right)^{k-1}  \frac{\lambda}{\lambda + 2} 
$$
The time $t^1$ for a transition $0 \to A \hbox{ or } B$ is exponential with rate $r_1 = \lambda$. The time $t^2$ for $A \hbox{ or } B \to 0 \hbox{ or } AB$ 
is exponential with rate $r_2 = 1 + (\lambda/2)$. The means are  $\E t^1 = 1/\lambda$ and $\E t^2=(1 + \lambda/2)^{-1} = O(1)$ respectively. As $\lambda \to 0$, the second will be much smaller than the first. The total waiting time is $(t^1_1+\cdots+t^1_N)+(t^2_1+\cdots+t^2_N)$.
The following well-known result shows each sum has an exponential distribution. We include its simple proof for completeness.

\begin{lemma} \label{sumgeom}
	If $N$ is geometric with success probability $p$, i.e., $P(N=k) = p (1-p)^{k-1}$, and $X_1, X_2, \ldots$ are an independent i.i.d.~sequence 
	with an exponential distribution with rate $r$ then $S_N = X_1 + \cdots + X_N$ is exponential with rate $pr$.
\end{lemma}

\begin{proof}
	$E\exp(\theta X_i) = r/(r+\theta)$ so conditioning on the value of $N$
	\begin{align*}
	E\exp(\theta S_N) & = p \sum_{k=1}^\infty (1-p)^{k-1} \left( \frac{r}{r+\theta} \right)^k \\
	& = \frac{ p r/(r+\theta) } { 1-  (1-p)r/(r+\theta) } \\
	& = \frac{pr}{r+\theta - (1-p)r } = \frac{pr}{\theta+pr}
	\end{align*}
	which is the Laplace transform of an exponential with rate $rp$ proving the desired result.
\end{proof}

Let $T^1_{AB} = t^1_1 + \cdots + t^1_N$ and $T^2_{AB} = t^2_1 + \cdots + t^2_N$. Using Lemma \ref{sumgeom} we see that $T^1_{AB}$ has an exponential distribution with rate 
$$
pr_1 = \lambda^2/(\lambda+2) \qquad\hbox{and mean }(\lambda+2)/\lambda^2.
$$ 
$T^2_{AB}$ has an exponential distribution with rate 
$$
pr_2=\lambda/2\qquad\hbox{and mean }2/\lambda.
$$
Let $c$ be a constant to be chosen later and  let
$$
T = c(\E T^1_{AB}+\E T^2_{AB} )= c\left[\frac{\lambda+2}{\lambda^2}+ \frac{2}{\lambda}\right] = c\frac{3\lambda+2}{\lambda^2}
$$ 
If we let $S^1_{-1}$ the time to produce an $AB$ at $-1$  then by Lemma \ref{sumgeom} we have
\beq
\P(S^1_{-1} > T) \leq \P(T^1_{AB}>c\E T_1^{AB})+ \P(T^2_{AB}>c\E T_2^{AB}) \leq 2e^{-c}
\label{ubc1}
\eeq
If we let  $S^1_{1}$ the time to produce an $AB$ at $1$ and $S^1_2$ be the additional time to produce an $AB$ at 2 then
\beq
\P(S^1_1+S^1_2> T) \le 2 \P( S^1_1 \ge T/2 )  \leq 4e^{-c/2}
\label{ybc2}
\eeq

To return to the situation where $\mu>0$ we note that the probability there is no $\mu$-death on $\{-1,0,1,2\}$ during $[0,T]$ is $e^{-4\mu T}$.
If we take $\mu = b/T$ then the probability of a $\mu$-death in the box is $1-e^{-4b}$. Thus, \xyz{to prove Lemma \ref{star}} we want to pick $c$ and $b$ so that 
\beq
2e^{-c} + 4e^{-c/2} + 1 - e^{-4b} < {1-0.726=0.274}
\label{ubgoal}
\eeq
To do this we first decide to take $c=9$ so that $2e^{-c} <  2.4682 \times 10^{-4}$, $4e^{-c/2} < 0.0444359$. This means we need 
$1- e^{-4b} \le 0.229317$. This holds if we take $b < b_0 = (1/4)\log(1/0.77069) = 0.028281$. Thus we have survival if
$$
\mu = \frac{b}{T } \le b_0 \cdot \frac{\lambda^2}{9(3\lambda + 2)} \quad\hbox{or}\quad \lambda^2 \ge \mu \frac{9(3\lambda+2)}{b+0}
$$
Hence by monotonicity of the survival probability we will have survival if
$$
\lambda^2 \ge \mu \frac{9(3\lambda+2)}{b_0}
$$
and $45\mu/b_0<1$ which holds if $\mu<1600$.
If $40\sqrt{\mu}>1$, then we can simply bound $\lambda_c(\mu)$ by $\lambda_c(1)$, the critical value for ordinary contact process on $\ZZ^1$.

\subsection{Lower bound on $\lambda_c$}

For some $0<\delta<1$ define $M_t=(AB)_t+\delta(A_t+B_t)$. Here $(AB)_t$, $A_t$ and $B_t$ are the number of sites in states $AB$, $A$, and $B$ at time $t$.
The idea is to show for certain values of $\lambda$, it's possible to choose the $\delta$ so that $M_t$ is a supermartingale. Since $M_t\ge 0$, it converges to a finite limit, which must be identically 0 since $M_t$ has values in a discrete set of values. This then implies both species have to die out with probability 1.
Changes in $M_t$ results from the following

\begin{itemize}
	\item An $AB$ becomes $A$ or $B$. This decreases $M_t$ by $1-\delta$ and the total rate is $2\mu (AB)_t$.
	\item $AB$ or $A$ or $B$ gives birth creating a new $AB$. This increases $M_t$ by $1-\delta$ and the total rate is $\leq 
	\lambda(A_t+B_t)$ (note an $A$ or $B$ can become an $AB$ at rate at most $\lambda$).
	\item An $AB$ gives birth so that a new $A$ or $B$ is created. This increases $M_t$ by $\delta$ and the total rate is $\leq 2\lambda (AB)_t.$
	\item An $A$ or $B$ dies. This decreases $M_t$ by $\delta$ and the total rate is $A_t+B_t$.
	\item An $A$ or $B$ creates a new $A$ or $B$. This increases $M_t$ by $\delta$. The total rate is $\leq\lambda(A_t+B_t).$
\end{itemize}

Based on the items in the previous list, we can conclude that if $F_t=\sigma\{A_s, B_s, (AB)_s, s\leq t\}$.
\begin{align*}
\frac{d}{dt}\E(M_t|F_t) &\leq -2\mu (AB)_t (1-\delta)+\lambda (A_t+B_t)(1-\delta)\\
& +2\delta\lambda (AB)_t-\delta(A_t+B_t)+\lambda (A_t+B_t)\\
&=[-2\mu(1-\delta)+2\delta\lambda](AB)_t+(\lambda(1-\delta
)-\delta+\lambda)(A_t+B_t),
\end{align*}

To get a supermartingale we fix $\lambda$ and $\mu$ and pick $\delta$ so that
$$
-2\mu(1-\delta)+2\delta\lambda< 0 \quad\hbox{and}\quad \lambda(1-\delta)-\delta+\lambda < 0.
$$
For this to hold we need
$$
\frac{\mu}{\mu +\lambda} > \delta > \frac{2\lambda}{\lambda+1}
$$
For this to be possible we need $\mu(\lambda+1) > 2\lambda(\mu+\lambda)$. Rearranging we want $0 > - \mu + \lambda\mu + 2\lambda^2$. 
Using the quadratic equation we find that a $\delta$ exists if and only if 
$$
\lambda \leq \frac{-\mu+\sqrt{\mu^2+8\mu}}{4}.
$$
This implies that $\lambda_c \ge \sqrt{\mu/2} - \mu/4$.

\subsection{Improved lower bound in $d=1$}

\xyz{
	Consider the symbiotic biased voter model (SBVM) starting with $AB$'s at all integers $n \le 0$. To be precise, the birth rates for $A$'s and $B$'s remains the same, however, in this model only the rightmost $A$ or $B$ is allowed to die. Let $r_A(t)$ be the right most $A$ and $r_B(t)$ the right most $B$. Each extreme particle dies at rate $\mu$ if the state there is $AB$ or at rate 1 if the site is singly occupied. For example, in the realization drawn in Figure 4, he $B$ at $r_B(t)$ dies at rate 1, while the $A$ at $r_A(t)$ dies at rate $\mu$. 
}

\begin{figure}[ht]
	\begin{center}
		\begin{picture}(200,80)
		\put(30,10){$A$}
		\put(45,10){$A$}
		\put(60,10){$A$}
		\put(75,10){$A$}
		\put(90,10){$A$}
		\put(105,10){$A$}
		\put(100,60){$r_A(t)$}
		\put(108,40){$\downarrow$}
		\put(30,25){$B$}
		\put(45,25){$B$}
		\put(60,25){$B$}
		\put(75,25){$B$}
		\put(90,25){$B$}
		\put(105,25){$B$}
		\put(120,25){$B$}
		\put(135,25){$B$}
		\put(150,25){$B$}
		\put(165,25){$B$}
		\put(160,60){$r_B(t)$}
		\put(168,40){$\downarrow$}
		\end{picture}
	\end{center}
	\caption{An example illustrating the defintions.}
\end{figure}
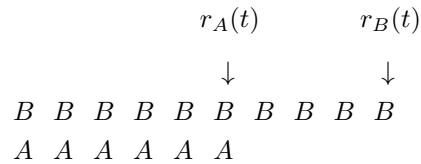

\xyz{Let $r(t)=\max\{r_A(t), r_B(t)\}$. The existence of an edge speed 
	$$
	\lim_{t\to\infty} r(t)/t) = \alpha(\lambda)
	$$
	can be proved using the subadditive ergodic theorem. Using arguments in \cite{D80} for the one-dimensional contact process (see also \cite{D84} for the case of oriented percolation) we can also show that if we define 
	$$
	\lambda_{edge}^{SBVM} = \inf\{ \lambda : \alpha(\lambda) > 0\},
	$$
	then the probability of survival for SBVM, starting from origin occupied by $AB$ and other sites vacant, is 0 for $\lambda < \lambda^{SBVM}_{edge}$. One can construct the SBVM on the same graphical representation of SCP by ignoring 
	deaths that do not effect the right-most $A$ and $B$. It follows from the coupling and the results quoted above that if $\lambda < \lambda_{edge}^{SBVM}$ then 
	the SCP dies out. From this it follows that $\lambda_{edge}^{SBVM}\leq \lambda_e^{SCP}$, the critical value for survival from a finite set defined in Section 1.3.}

If $r_A(t) < r_B(t)$ then the right-most $A$ and the right most $B$ give birth at rate $\lambda/2$. Let $N_t=|r_A(t)-r_B(t)|$. If $N_t = n>0$ and $r_A(t) < r_B(t)$ then the $A$ at $r_A(t)$ dies at rate 1, and the $B$ at $r_B(t)$ dies at rate $\mu$.
Let $q$ be the Q-matrix for $N_t$.

\begin{itemize}
	\item 
	If $n \ge 1$ then $q(n,n+1)=\mu +\lambda/2$ due to birth of an $A$ or death of a $B$; \hbr $q(n,n-1)=1+ \lambda/2$ due to death of an $A$ or birth of a $B$. 
	\item
	If $n=0$ then $q(0,1) = \lambda + 2\mu$ due to a birth or death of an $A$ or a $B$.
\end{itemize}

To compute the stationary distribution set $\mu(1)=c$. If $n > 1$
$$
\pi(n) \left(1+\frac{\lambda}{2}\right) = \pi(n-1) \left(\frac{\lambda}{2}+\mu \right)
$$
so we have
\beq
\pi(n) = c \left( \frac{\lambda+2\mu}{\lambda+2} \right)^{n-1}
\label{pi1}
\eeq
To compute $\pi(0)$ we note that
\beq
\pi(0)(\lambda + 2\mu) = c \left(1+\frac{\lambda}{2}\right)
\label{pi2}
\eeq
Using \eqref{pi1} and \eqref{pi2} then summing we have
\begin{align*}
\sum_{n=0}^\infty \pi(n) & = \frac{c ( 1+\lambda/2)}{\lambda + 2\mu} + \frac{ c }{ 1 - (\lambda+ 2\mu)/(\lambda+2) } \\
& = c \left[ \frac{1+\lambda/2}{\lambda+2\mu} + \frac{\lambda+2}{2(1-\mu)} \right]
\end{align*}
From this and our bullet list of flip rates, it follows that in equilibrium
\begin{align*}
\frac{d}{dt} \E r(t) & = c \left[ \frac{1+\lambda/2}{\lambda+2\mu} (\lambda-2\mu) 
+ \frac{\lambda+2}{2(1-\mu)} (\lambda/2-1) \right] \\ 
& = c \left(1+\frac{\lambda}{2}\right)\left[ \frac{\lambda-2\mu}{\lambda+2\mu}  + \frac{\lambda/2-1}{1-\mu}  \right]
\end{align*}
To find when this is positive we set
$$
(\lambda-2\mu)(1-\mu) + (\lambda+ 2\mu)(\lambda/2-1) = 0
$$
A little algebra converts this into
\begin{align*}
0 &= 2\mu^2 -(\lambda+2) \mu + \lambda + \mu(\lambda-2)  + \lambda^2/2 - \lambda \\
&= 2\mu^2 - 4\mu + \lambda^2/2
\end{align*}
so for the SBVM the critical value $\lambda_c^{SBVM} =\sqrt{8\mu-4\mu^2}$. As remarked \xyz{above this} is a lower bound on $\lambda_c(\mu)$ for SCP. 

\clearp

\section{Proof of Lemma \ref{block_new}} \label{sec:pfL1}

At several points in the proof we will use Harris' result on ``positive correlations.'' See Theorem B17 on page 9 of \cite{TL}.
If $\zeta_t$ is an attractive particle system starting from a deterministic initial condition and $\mu$ is the distribution at time $t$ then for any increasing functions $f$ and $g$
\beq
\int f g \, d\mu \ge \int f \, d\mu \cdot \int g \, d\mu 
\label{pc}
\eeq
In most of our applications $f=1_A$ and $g=1_B$ will be indicator functions.

We begin our investigation by recalling some facts about the subcritical $d$-dimensional contact process.  
Let $\lambda$ be the total birth rate from an occupied site and let the death rate be 1. We suppose that in the initial condition
only the origin is occupied and denote the process by $\eta^0_t$. Combining Theorems 2.34 and 2.48 of \cite{TL} shows 
that if $\lambda < \lambda_c$ then there exists some $\kappa(\lambda)<1$ so that 

\begin{equation}\label{expdecay}
\mathbb{E}(\abs{\eta_t}) \leq C\kappa(\lambda)^t,
\end{equation}
for all $t$. 
By increasing the value of $\kappa(\lambda)$ if necessary we also have
\begin{equation}\label{expdecay2}
\P(\eta^{0}_t  \mbox{ ever reaches }\partial B_1(0,s) \mbox{ for some $t$})\leq C\kappa(\lambda)^s,
\end{equation}
where $\partial B_1(0,s)$ is the boundary of the $L^1$ ball of radius $s$ in $\ZZ^d$ centered at 0. 
Here $L^1$ means $dist(x,y)=\sum_{i=1}^d \abs{x_i-y_i}$ for $x=(x_1,\ldots, x_d)$ and $y=(y_1, \ldots, y_d)$. 

Let $a$ be chosen so that
\begin{equation}\label{acond}
(\kappa(\lambda))^a  e^8 (1-\exp(-\frac{\lambda}{2d}))^{-4} < 1.
\end{equation}
The reason we choose an $a$ that satisfies \xyz{the} above will become clear in the proof of Lemma \ref{ABinsides}.

Recall that $R_{L_j,T_j} = \{ AB(H_{L_j,T_j}) \le M, AB(S_{T_j,L_j}) \le {N} \}$.
Define $D_{L_j,T_j,l}$ to be the event that there are fewer than $l$ arrows with one side occupied by $A$ or $B$ or $AB$ escaping from the side of space time box, i.e., $\sigma_{L_j,T_j}$. Define $E_{L_j,T_j,l}$ to be the event that there are fewer than $l$ sites \xyz{in state $A$ or $B$ or $AB$} on the top,
$W(L)\times \{1.01T_j\}$.

\begin{lemma}\label{noflee}
	Let $M$ and $N$ be fixed. For any $T_j, L_j \to \infty$ such that $\lim_{j\to \infty} T_j (\kappa(\lambda))^{aL_j}L_j^{d-1}=0$ 
	and $\lim\limits_{j\to \infty}\frac{T_j}{\log L_j}=\infty$  
	we have
	\begin{align}
	\lim_{l\to \infty} \limsup_{j\to \infty}\P(R_{L_j,T_j}\cap D^c_{L_j,T_j,l}) =0
	\label{dlim} \\
	\lim_{l\to \infty} \limsup_{j\to \infty}\P(R_{L_j,T_j}\cap E^c_{L_j,T_j,l})=0.
	\label{elim}
	\end{align}
\end{lemma}

\begin{proof}
	\xyz{To begin we consider equation \eqref{dlim}. 
		Observe that for any $A$ or $B$ reaching the outer boundary $\sigma_{L_j,T_j}$, that particle must have an ancestor 
		on the inner boundary $\rho_{L_j,T_j}$ s.t. the infection path
		from the ancestor to the descendant completely lies in the side slab
		$S_{T_j,L_j}$. There are two possibilities for this infection path. Either it meets a particle of the other species during the journey or not. We say the descendants of an $A$ travel alone if they never share a site with a $B$ individual.} \xyz{An individual $A$ on $\rho_{L_j,T_j}$  has probability of less than $C(\kappa(\lambda))^{aL_j}$ to reach $\sigma_{L_j,T_j}$ if they travel alone and the infection path stays in the side slab. Note that due to their definitions $A(\rho_{L_j,T_j})$ and $B(\rho_{L_j,T_j})$ are bounded by $CT_jL_j^{d-1}$.  Since we assume  
		$$ 
		\lim_{j\to \infty} T_j(\kappa(\lambda))^{aL_j}L_j^{d-1}=0,
		$$ 
		by union bound the probability that there exists one $A$ or $B$ on $\rho_{L_j,T_j}$ that make it to $\sigma_{L_j,T_j}$
		by traveling alone in the side slab
		is bounded by $C'T_jL_j^{d-1} (\kappa(\lambda))^{aL_j}$, which goes to 0 as $j\to \infty$ by assumption. Since $N$ is fixed finite number on the event $R_{L_j,T_j}$ the number of particles in the side slab that do not travel alone in the side slab is at most $C''N$ with high probability.
		Combining the two types of possibilities for spreading (traveling alone or not) we see the number of offspring that can reach the sides
		$\sigma_{L_j,T_j}$
		is also at most $C''N$ with high probability. 
		This proves \eqref{dlim}. The proof of \eqref{elim} is similar.} 
\end{proof}

\begin{proof}[Proof of Lemma \ref{block_new}] For any given $\eta>0$, by Lemma \ref{noflee} we can pick $l$ large enough s.t., 
	$$
	\limsup_{j\to \infty}\P(R_{L_j,T_j})\leq \eta+\limsup_{j\to \infty}\P(R_{L_j,T_j}\cap D_{L_j,T_j,l}\cap {E_{L_j,T_j,l}} ).
	$$
	Each individual dies at rate at least $\mu$, so the probability for it to die in 1 unit of time is at least $1-\exp(-\mu)$. The probability that this individual does not produce a new particle within 1 unit of time is at least $\exp(-\lambda)$. Using positive correlations, we see that on the event $D_{L_j,T_j,l}\cap {E_{L_j,T_j,l}}$, there is always probability 
	$$ 
	\ge [ (1-\exp(-\mu))\exp(-\lambda)   ]^{2l}
	$$
	that all the births escaping from the side are killed and all the sites occupied at top slice recover before giving any birth. 
	If we denote by ${\cal F}_{B(L_j,T_j)}$ the sigma algebra generated by the Poisson process in the space-time box $B(L_j,T_j)$,  
	then on $D_{L_j,T_j,l}\cap E_{L_j,T_j,l}$ we have 
	$$
	\P(\Omega_0 | {\cal F}_{B(L_j,T_j)})\ge [ (1-\exp(-\mu))\exp(-\lambda)   ]^{2l}.
	$$
	\xyz{
		L\/evy's 0-1 law 
		(see e.g. \cite[Theorem 5.5.8]{PTE}) implies that
		$$
		P( \Omega_0|{\cal F}(L_j,T_j)) \to 1_{\Omega_0},
		$$
		which implies on the event $\limsup_{j\to \infty} D_{L_j,T_j,l} \cap E_{L_j,T_j,l}$ we have
		$1_{\Omega_0}=1$.
		It follows from this and the fact  
		$
		\limsup_{j\to \infty}\P(D_{L_j,T_j,l}\cap E_{L_j,T_j,l}) \leq 
		\P(\limsup_{j\to \infty} D_{L_j,T_j,l} \cap D_{L_j,T_j,l})
		$
		that for any $l$,
		$$
		\limsup_{j\to \infty}\P(D_{L_j,T_j,l}\cap E_{L_j,T_j,l}) \leq \P(\Omega_0).
		$$
		Lemma \ref{block_new} follows because $\eta$ can be taken  arbitrarily small.}
\end{proof}

\clearp 

\section{Proof of Lemma \ref{persis_new}} \label{sec:pfL2}

In this section, our goal is to prove

\mn
{\bf Lemma 1.7} {\it For any sequence $L_j\to \infty$, one can choose $T_j=T(L_j)$ that satisfies 
	\begin{equation}\label{tj}
	\lim_{j\to \infty} T_j (\kappa(\lambda))^{aL_j}L_j^{d-1}=0 \mbox{ and } \liminf_{j\to\infty} \frac{T_j}{L_j}>0,
	\end{equation} 
	so that for $j$ large  $\P( AB(S_{L_j,T_j}) \ge N )=1-\sqrt{\P(\Omega_0)}.$}

\mn
Let $J_{L,t} = \{ AB(S_{L_j,t}) \ge N \}$ and let $f(L_j,t) = \P(J_{L_j,t})$. The first step is to show 

\begin{lemma}\label{ABinsides}
	\xyz{Let $L_j$ be any sequence going to $\infty$.}  
	Set $c_3=e^{-8}(1-e^{-\lambda/(2d)})^4$ and define
	$$
	t_j=\left[\frac{1}{c_3(\kappa(\lambda))^a} \right]^{L_j/2},
	$$
	then \begin{equation}\label{noABdeath}
	\lim_{j \to \infty}(f(L_j,t_j)-\P(J_{L_j,t_j}, {}_{W(L)}\xi_{t_j}=\emptyset))=0
	\end{equation}
	and 
	\beq \label{tjub}
	\lim_{j\to \infty} t_j(\kappa(\lambda))^{aL_j}L_j^{d-1}=0.
	\eeq
	
\end{lemma}

To prove this we use the following proposition:

\begin{proposition}\label{twoobs}
	For any $t$ and any possible configuration ${}_{W(L_j)} \xi_t$, there are two possibilities
	
	\begin{itemize}
		\item {\bf Scenario 1.} \xyz{
			There is no $AB$ in ${}_{W(L_j)} \xi_t$ and
			for any $x, y\in W(L_j)$ with an $A$  at $x$ and a $B$  at $y$}
		we have $dist(x,y)\geq L_j$. In this case there are constants $c_1,c_2$ so that	
		\begin{equation}\label{obs1}
		\P_{{}_{W(L_j)} \xi_t  }({}_{W(L_j)}\xi_{t+L_j} =\emptyset)\geq  1-c_1((a+h)L_j)^d\exp(-c_2L_j).
		\end{equation}
		
		\item  { \bf Scenario 2.} There are $x, y$ with $dist(x,y)\leq L_j$ with an $A$ or $AB$ at $x$ and a $B$ or $AB$ at $y$. 
		If we let  $c_3=e^{-8}(1-e^{-\lambda/2})^4$ and we let $U_{L_j,t} = [-(a+h)L_j, -hL_j]\times I^{d-1}\times  [t,t+5L_j] \cup [hL_j, (a+h)L_j]\times I^{d-1} \times [t,t+5L_j]$ where $I=[-(a+h)L_j, (a+h)L_j]$
		then
		\beq \label{obs2}
		\P_{{}_{W(L_j)} \xi_t} ({}_{W(L_j)}\xi \mbox{ has an AB in $U_{L_j,t}$ })>c_3^{L_j}
		\eeq
	\end{itemize}
\end{proposition}

\begin{proof}[Proof of Proposition 1]
	In scenario 1, equation \eqref{expdecay2} implies that a single particle $x$ has probability at most $C\kappa(\lambda)^{L_j/2}$ to ever reach some point outside $B_1(x,L_j/2)$.  Using a union bound, we see that with probability at least $1-C((a+h)L_j)^d \kappa(\lambda)^{L_j}$ none of the particles in the box will move for a distance greater than $L_j/2$. Conditioned on the event that none of the $A$'s and $B$'s meet, they will die out before time $L_j$ with a probability of
	at least $1-C((a+h)L_j)^d\kappa(\lambda)^{L_j}$ by equation \eqref{expdecay}. Using union bound again we get equation \eqref{obs1}. 
	
	In Scenario 2, the result comes from the fact that for any $A$ and $B$ with distance less than $L_j$, the probability that they will meet to form an $AB$ within time of $L_j$ is 
	$$
	\ge (e^{-1} ((1-e^{-\lambda/(2d)} ) e^{-1})^{L_j}. 
	$$
	To prove this, let $z_0=x, z_1, \ldots z_k=y$ be a path with length $\le L_j$ and note that to get from $z_i$ occupied at time $t$ to $z_{i+1}$ occupied at time $t+1$: the particle at $z_i$ has to survive for time 1, give birth onto $z_{i+1}$ and the newborn particle needs to survive until time $t+1$.  
	If the $AB$ is already in $[-(a+h)L_j, -hL_j]\times I^{d-1} \cup [hL_j, (a+h)L_j] \times I^{d-1}$ then we are done. Otherwise we let this $AB$ produce an $AB$ in $[-(a+h)L_j, -hL_j]\times I^{d-1} \cup [hL_j, (a+h)L_j] \times I^{d-1}$ within time $2L_j$ with probability at least 
	$$
	(e^{-2\mu}(1-e^{-\lambda/(2d)})^2 e^{-1})^{L_j}.
	$$ 
	To see this, for any point in $[-(a+h)L_j,(a+h)L_j]$ there is a path of length at most $L_j$ connecting this point to some other point in
	$[-(a+h)L_j, -hL_j]\times I^{d-1} \cup [hL_j, (a+h)L_j] \times I^{d-1}$. \xyz{This time
		the events along the path that result in spreading the infection are: 
		
		(i) The $AB$ at $z_i$ has to survive for time 1 (with probability $e^{-2\mu}$). 
		
		(ii) During this time unit there has to be a birth of an $A$ and a $B$ at $z_{i+1}$ which occurs with probability $\geq (1-e^{-\lambda/(2d)})^2$). 
		
		(iii) The first born particle has to survive until the second birth occurs which occurs with probability $\geq e^{-1}$). 
		
		\noindent
		Then we repeat this process for the newly formed $AB$ until we reach the final point of the path.
		Using positive correlations we prove \eqref{obs2} with $c_3=e^{-8}  (1-e^{-\lambda/(2d)} )^4$ since
		$(e^{-1} ((1-e^{-\lambda/(2d)} ) e^{-1})^{L_j} (e^{-2\mu}(1-e^{-\lambda/(2d)})^2 e^{-1})^{L_j}>c_3^{L_j}$.}
\end{proof}

\begin{proof}[Proof of Lemma \ref{ABinsides}] We divide $[0,t_j]$ into intervals of length $5L_j$.
	Denote by $t_{j,k}$ the resulting division points for $1\leq k\leq t_j/(5L_j)$
	By Proposition \ref{twoobs}, on the event ${}_{W(L_j)}\xi_{t_j} \neq \emptyset$
	if we look at the restricted process ${}_{W(L_j)}\xi_t$ at times $t_{j,,k} $ we need
	to stay in scenario 2 otherwise the restricted process will die with high probability. On the other hand, falling into scenario 2
	implies that we will have a chance of at least $c_3^{L_j}$ to get an $AB$ in $S_{L,T}$. If  $t_j/L_j$ is large enough so that
	\begin{equation}\label{choose tj}
	\lim_{j\to\infty} \frac{t_j \cdot c_3^{L_j}}{L_j} =\infty,
	\end{equation}
	then with high probability the number of $AB$ in $S_{L_j,t_j}$ will grow to $\infty.$ To see this,
	recall that for any collection of independent Bernoulli random variables $ \{I_r\}$ 
	using Chebyshev's inequality with the fact that $\var(I_r) \le \E(I_r)$ we have 
	$$
	\P\left(\sum_{r} I_r \geq \frac{1}{2}\sum_{r} \E(I_r) \right)\geq 1-\frac{4}{\sum_{r} \E(I_r)}.
	$$
	If we choose  
	$
	t_j=[1/(c_3\kappa(\lambda))^a ]^{L_j/2},
	$
	equation \eqref{choose tj} is satisfied since $c_3/\kappa(\lambda)^a>1$, which implies
	$$
	\lim\limits_{j\to \infty}\P({}_{W(L_j)} \xi_{t_j}\neq \emptyset,AB( S_{L_j,t_j})\leq N)=0.
	$$
	This proves \eqref{noABdeath}. Since $\kappa(\lambda)^a/c_3<1$ it is clear from the definition of $t_j$ that \eqref{tjub} holds.
\end{proof}

\begin{proof}[Proof of Lemma \ref{persis_new}]
	It is  clear that $f(L_j,0)=1$ and $f(L_j,t)$ is an decreasing function of $t$. We need to find $T_j$ that satisfies $f(L_j,T_j)=\sqrt{\P(\Omega_0)}$. By comparison with Richardson's model, see \cite{Rich}, it follows that the SCP spreads at most linearly. Hence if $T_j=o(L_j)$ then 
	$\lim_{j \to \infty}f(L_j,T_j)=1$. This implies that if we have
	$$
	f(L_j,T_j)=\sqrt{\P(\Omega_0)}, \mbox{for large } j,
	$$
	then it must be true that $\liminf_{j\to\infty} T_j/L_j >0.$
	
	Lemma \ref{ABinsides} gives us, see \eqref{tjub} and \eqref{noABdeath}, a sequence of $t_j$ that satisfies $\lim_{j\to\infty} t_j (\kappa(\lambda))^{aL_j}=0$ and
	$\limsup_{j\to \infty}f(L_j,t_j) \leq  \P(\Omega_0)$. Since $f(L_j,t)$ is decreasing in $t$ we have proved the existence of $T_j$ with the properties desired in  Lemma \ref{persis_new}. 
\end{proof}

\clearp 	

\section{Proof of Theorem \ref{opstep1}} \label{sec:pfT2}

Using Lemma \ref{block_new}, positive correlations, and Lemma \ref{persis_new} we see that if $j$ is large
\begin{align*}
2\P(\Omega_0) &\ge P(  AB(H_{L,T}) \le  M, AB(S_{T,L}) \le N  ) \\
& \ge \P(  AB(H_{L,T}) \le  M) \P( AB(S_{T,L}) \le N  ) \\
& = \P(  AB(H_{L,T}) \le  M) \sqrt{\P(\Omega_0)}
\end{align*}
Rearranging we have
\beq
\P(  AB(H_{L,T}) >  M) \ge 1 - 2 \sqrt{\P(\Omega_0)}
\label{optop}
\eeq
For the moment we will restrict to $d=1$ and drop the superscript $i$ from the notation for the slabs. 
To bound the other probability we use Lemma \ref{persis_new}, monotonicity, positive correlations, 
and symmetry to conclude  
\begin{align*}
\sqrt{\P(\Omega_0)} = P( AB(S_{T,L}) \le N  )& \ge \P( AB(S^+_{T,L}) \le N/2,  AB(S^-_{T,L}) \le N/2 ) \\
& \ge \P( AB(S^+_{T,L}) \le N/2) \P(  AB(S^-_{T,L}) \le N/2 ) \\
& = \P( AB(S^+_{T,L}) \le N/2)^2
\end{align*}
which gives
$$
\P( AB(S^+_{T,L}) > N/2),  \P( AB(S^-_{T,L}) > N/2) \ge 1 - \P(\Omega_0)^{1/4}
$$
The same calculation \xyz{gives for the SCP on $\ZZ^d$}:
\beq
\P( AB(S^{i,+}_{T,L}) > N/2),  \P( AB(S^{i,-}_{T,L}) > N/2) \ge 1 - \P(\Omega_0)^{1/4d}
\label{oplr}
\eeq

It remains to show if we have many $AB$'s then with high probability at some time $t \le 1.01T$ we will have an box of length $2n$ that is filled with $AB$. 
Denote by $F_1$ the event that ${}_W\xi^{[-n,n]^d}_{t} \supset x+[-n,n]^d$ for some $x\in [hL,(a+h)L] \times I^{d-1}$ and some $t\in [0,1.01T]$,
and $F_2 = \{  AB(S^{1,+}_{T,L}) > N/2\} $. Let $\ep$ be the constant in Theorem \ref{opstep1} and let $\delta = (\ep/4)^4$.
Note that by the ergodic theorem we know that $\P(\Omega_0) \to 0$ as $n\to \infty$, 
so if we pick $n$ large enough then $\P(\Omega_0) \le \delta/2$. Having chosen the value of $n$, we will show

\begin{lemma} \label{claim} 
	$\P(F^c_1\cap F_2) \leq \delta^{1/4}$ for $N$ large enough.
\end{lemma}

Assume Lemma \ref{claim} for the moment, it follows from Lemma \ref{persis_new}
$$
\P(F_1^c)\leq \P(F_2^c)+\P(F_1^c\cap F_2)\leq 2\delta^{1/4} = \ep,
$$
Equation \eqref{topop} can be shown similarly by using \eqref{optop} in place of \eqref{oplr}.

\begin{proof}[Proof of Lemma \ref{claim}.] Following the proof of (5.6) in \cite{BG},
	we will use an algorithm that will stop if we find $(2n)^d$ consecutive $AB$'s. It is designed so that each time we choose an $AB$ site,  the probability that this chosen point fails to generates $(2n)^d$ $AB$'s within next one unit of time is less than $\alpha$ which is a number only depends on $n$. Also whether or not we can get $(2n)^d$ $AB$'s is independent of information obtained before. This proves that $\P(F^c_1\cap F_2) \leq 2(2\delta)^{\frac{1}{4}}$ as long as the algorithm allows us enough choices on $F_2$.
	
	Let $t_1$ be the first time that an $AB$ appears in $S^{1,+}_{L,T}$ and let the coordinate of this first point by $(x_1,t_1)$.
	Let ${\cal F}_1$ be the sigma-algebra generated by the Poisson processes in $W$ up to time $t_1$. Given ${\cal F}_1$, the probability that this chosen point fails to generates an interval of $(2n)^d$ $AB$'s in $(x_1+[-n,n]^d)\times [t_1,t_1+1]$ using only Poisson arrivals in $(x_1+[-n,n]^d)\times [t_1,t_1+1]$ is less than $\alpha$. 
	If we do get $(2n)^d$ AB's then the algorithm stops otherwise we continue.
	
	At the next step we try to find the first $t_2 > t_1$, so that we have an $AB$ at $(x_2,t_2)$ in  $S^{1,+}_{L,T}$. When $t_1 \le t_2 \le t_1+1$ we ignore points with $|x_2-x_1| \le 4n$. If $t_2>t_1+1$ we let ${\cal F}_1$ be the sigma-algebra generated by the Poisson processes in $W$ up to time $t_2$. If $t_1 \le t_2 \le t_1+1$, we need to add information about the Poisson processes in $(x_1+[-n,n]^d)\times [t_1,t_1+1]$. In either case, given ${\cal F}_2$, there is a probability of less than $\alpha$ that we cannot find $(2n)^d$ consecutive $AB$'s by using the Poisson process in $ (x_2+[-n,n]^d) \times [t_2,t_2+1]$.

	We continue the search until we either get $(2n)^d$ consecutive $AB$'s or we come to time $1.01T$.
	The probability of failure when running this algorithm for $k$ steps is $\le \alpha^k$. This will be $\le \delta^{1/4}$ if
	$$
	k > (1/4) \frac{\log(\delta)}{\log(\alpha)}.
	$$
	Recall that we count two space time points $(y,s_1)$ and $(y,s_2)$ as different $AB$ only if $\abs{s_1-s_2}\geq 1$.  Hence at each step of the algorithm we ignore at most $(8n)^d$ $AB$ points.  It follows that if
	$$
	N>(1/4) \frac{\log \delta}{\log \alpha}((8n)^d+1),
	$$
	then we can run the algorithm for at least $1/4(\log \delta)/(\log \alpha)$ steps if needed, which implies the probability of failure is $\leq \delta^{1/4}$. This completes the proof of Lemma 7 and hence completes the proofof Theorem 2.
\end{proof}

\clearp

\section{Proof of Theorem \ref{SCPD}}\label{sec:thm3}

The first step is define the dual process. For this we need to construct the process using Poisson processes. For each ordered pair of neighboring sites $x,y \in \ep\ZZ^d$ and level $i$ we have a rate $\lambda$ Poisson process $T^{(x,y),i}_n$, $n \ge 1$ which cause births from $(x,i)$ to $(y,i)$. For each unordered pair of neighbors and level $i$ we have a rate $\ep^{-2}$ Poisson process $S^{x,y,i}_n$, $n \ge 1$ the values at $(x,i)$ to $(y,i)$. For each site $x$ and level $i$ we have a rate $\mu$ Poisson process $T^{(x,i)}_n$, $n \ge 1$ that always causes death of a particle at $x$ on level $i$, and a rate $1-\mu$ Poisson process $S^{(x,i)}_n$, $n \ge 1$ that causes death of a particle at $x$ on level $i$ if there is no particle at $x$ on level $1-i$.

The first three Poisson processes are part of an additive process but the last one is not, so the dual process that works backwards from $x$ on level $i$ at time $t$ is what \cite{DC} call the influence set. The state at $(x,i)$ at time $t$ can be computed if we know the states of the sites in ${\cal I}^{x,i,t}_s$ at time $t-s$. When a particle in  ${\cal I}^{x,i,t}_s$ hits an arrival in one of our Poisson processes the following changes occur.

\begin{itemize}
	\item $T^{(x,y),i}_n$. We draw an arrow from $(y,i) \to (x,i)$ to indicate a potential birth and add a particle at $(y,i)$. If there is already a particle at $(y,i)$ the two particle coalesce to one.
	\item $S^{x,y,i}_n$. We draw an arrow from $(x,i) \to (y,i)$ and an arrow from $(y,i) \to (x,i)$ to indicate that the values will be exchanged. We move the particle at $(x,i)$ to $(y,i)$. if there is a particle at $(y,i)$ it moves to $(x,i)$ 
	\item $T^{x,i}_n$. We kill the particle at $(x,i)$ and write a $\delta$ next to it.
	\item $S^{x,i}_n$. We write a $\delta$ at $(x,i)$ and draw an arrow from $(x,1-i) \to (x,i)$ to indicate that if $(x,1-i)$ is occupied the particle at $(x,i)$ is saved from death. We leave $(x,i)$ in the dual and add a particle at $(x,1-i)$. 
\end{itemize} 

For more details about the construction of the dual see Section 2 in \cite{DC}. In that section it is shown that the correlations of particle movements caused by stirring neighboring particles tend to 0 as $\ep\to0$, so the dual converges to a branching Brownian motion. In Sections 2d and 2e of \cite{DC}, the convergence of the dual to branching Brownian motion is used to conclude that the $q^\ep_A(t,x)$ and $q^\ep_B(t.x)$ converge to limits $q_A(t,x)$ and $q_B(t,x)$ that satisfy  \eqref{evoqA} and \eqref{evoqB}. This is also proved in Chapter 2 of \cite{CDP}. However the fact that we have stirring instead of random walks makes things simpler: we do not have to trim the dual to remove particles that exist for only  a very short amount of time.

The results of Aronson and Weinberger \cite{AW75,AW78} imply that when $\lambda>1$ Lemma 3.2 in \cite{DC} holds and consequently there is a nontrivial stationary distribution. Essentially the same proof is carried out in Chapter 6 of \cite{CDP}, but in that reference \xyz{a more careful comparison with percolation is used} to guarantee that the densities are close to the values that emerge from the mean-field ordinary differential equation. 

To prove result (ii) we have to use a comparison with oriented percolation to show that holes (regions that are $\equiv 0$) grow linearly so the system dies out. In doing this we have to deal with the fact that our block event that produces dead regions sometimes fails. \xyz{To guarantee} that the configuration is $\equiv 0$ even in that case, we use another percolation argument to show that the failed block events are surrounded by a connected dead region so it is impossible for there to be particles in the region with the failed block construction. This argument is done in Sections 4 and 5 of \cite{DC} for the quadratic contact process, and in greater generality in Chapter 7 of \cite{CDP}.  




\ACKNO{We would like to thank the referees for their helpful comments which improved the clarity of our presentation.  }


\end{document}